\documentclass[10pt]{article}

\usepackage{amsmath, amsthm, amssymb, commath, enumerate}
\usepackage{caption, graphicx, subcaption}
\usepackage[text={6.5in,9in},centering]{geometry}
\usepackage[numbers, comma, square, sort&compress]{natbib}
\usepackage[hypertexnames=false, colorlinks=true, linkcolor=blue, citecolor=red, urlcolor=cyan]{hyperref}

\graphicspath{{./Figures}}
\newtheorem{theorem}{Theorem}[section]
\newtheorem{hypothesis}{Hypothesis}

\newtheorem{corollary}[theorem]{Corollary}
\newtheorem{definition}[theorem]{Definition}
\newtheorem{lemma}[theorem]{Lemma}

\newtheorem{remark}[theorem]{Remark}
\numberwithin{equation}{section}

\newenvironment{Acknowledgment}%
{\begin{trivlist}\item[]\textbf{Acknowledgments }}{\end{trivlist}}

{\begin{trivlist}\item[]\textbf{Proof#1 }}%
{\hspace*{\fill}$\rule{0.3\baselineskip}{0.35\baselineskip}$\end{trivlist}}

\makeatletter\@addtoreset{equation}{section}\makeatother

\DeclareMathOperator{\Fix}{Fix}

\DeclareMathOperator{\Ran}{Ran}

\DeclareMathOperator*{\argmin}{arg\,min}

\title{Truncation of contact defects in reaction-diffusion systems}
\author{Milen Ivanov\footnote{Institute of Mathematics and Informatics, Bulgarian Academy of Sciences},~ Bj\"orn Sandstede\footnote{Division of Applied Mathematics, Brown University} }
\begin{document}
\maketitle

\begin{abstract}
Contact defects are time-periodic patterns in one space dimension that resemble spatially homogeneous oscillations with an embedded defect in their core region. For theoretical and numerical purposes, it is important to understand whether these defects persist when the domain is truncated to large spatial intervals, supplemented by appropriate boundary conditions. The present work shows that truncated contact defects exist and are unique on sufficiently large spatial intervals.\\

\end{abstract}

\section{Introduction}
\label{section:overview}

Solutions of reaction-diffusion systems exhibit a wide variety of patterns, which makes them ubiquitous in modeling chemical, biological and ecological models \citep{murray2007mathematical}. For example, Turing patterns are potential mechanism for the emergence of stripes and spots on animal coats \citep{turing1952chemical}. In chemistry, spontaneous pattern generation occurs in experiments of the Belousov--Zhabotinsky reaction \citep{yoneyama1995wavelength} and in numerical simulations of model systems, in which both rigidly-rotating spiral waves and spiral waves exhibiting one or more line defects have been observed (see Figure~\ref{fig:spiralwaves}).

\begin{figure}[t]
	\centering
	\begin{subfigure}[b]{0.2\textwidth}
		\centering
		\includegraphics[height=0.9\textwidth]{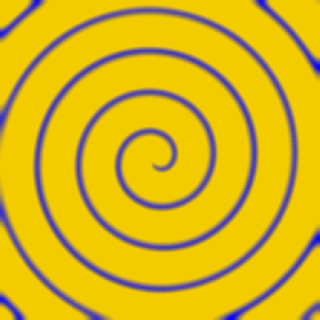}
		\caption{}\label{fig:spiral_wave}
	\end{subfigure}
	\begin{subfigure}[b]{0.2\textwidth}
		\centering
		\includegraphics[height=0.9\textwidth]{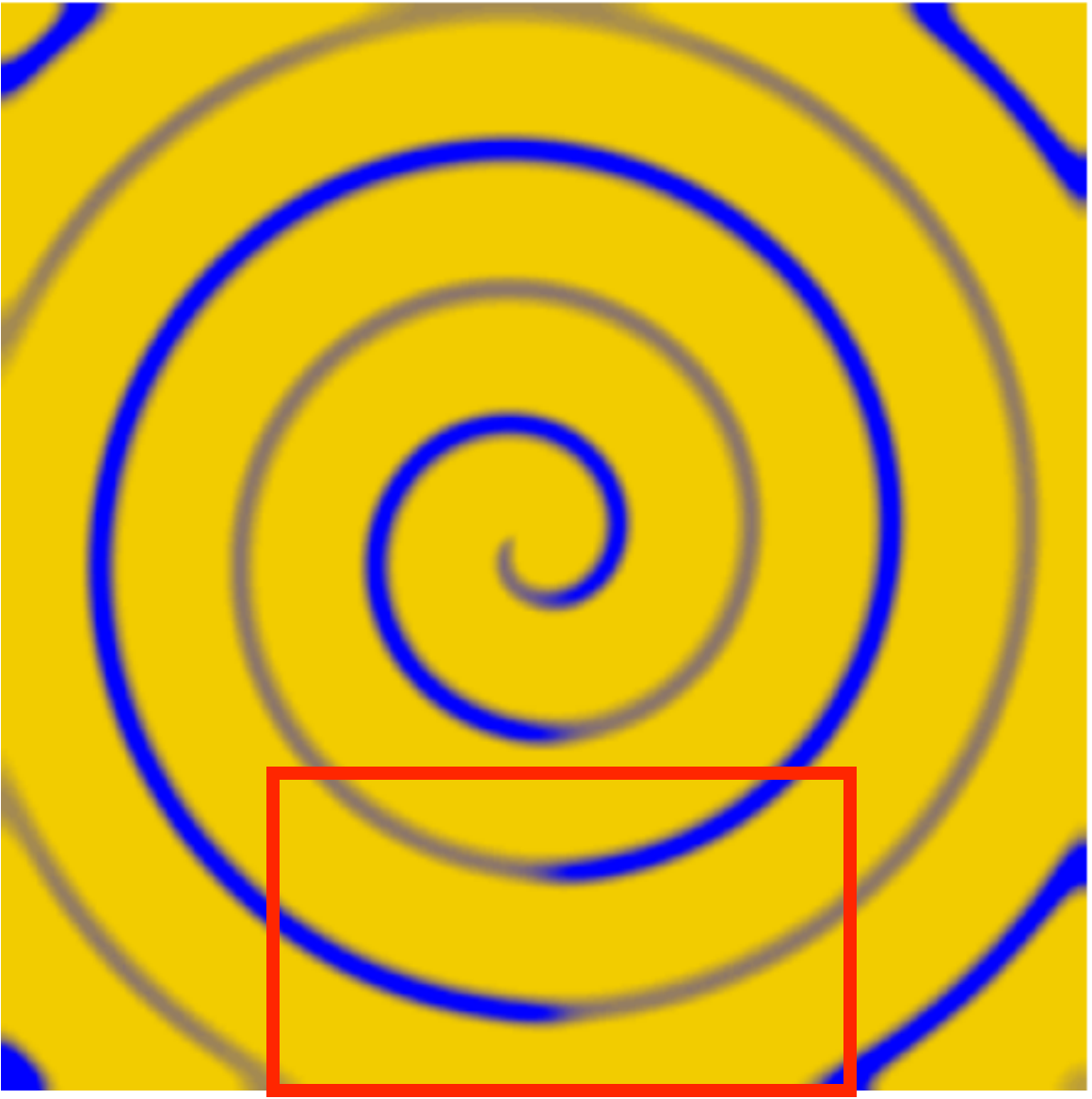}
		\caption{}\label{fig:spiral_defect}
	\end{subfigure}
	\begin{subfigure}[b]{0.2\textwidth}
		\centering
		\includegraphics[height=0.9\textwidth]{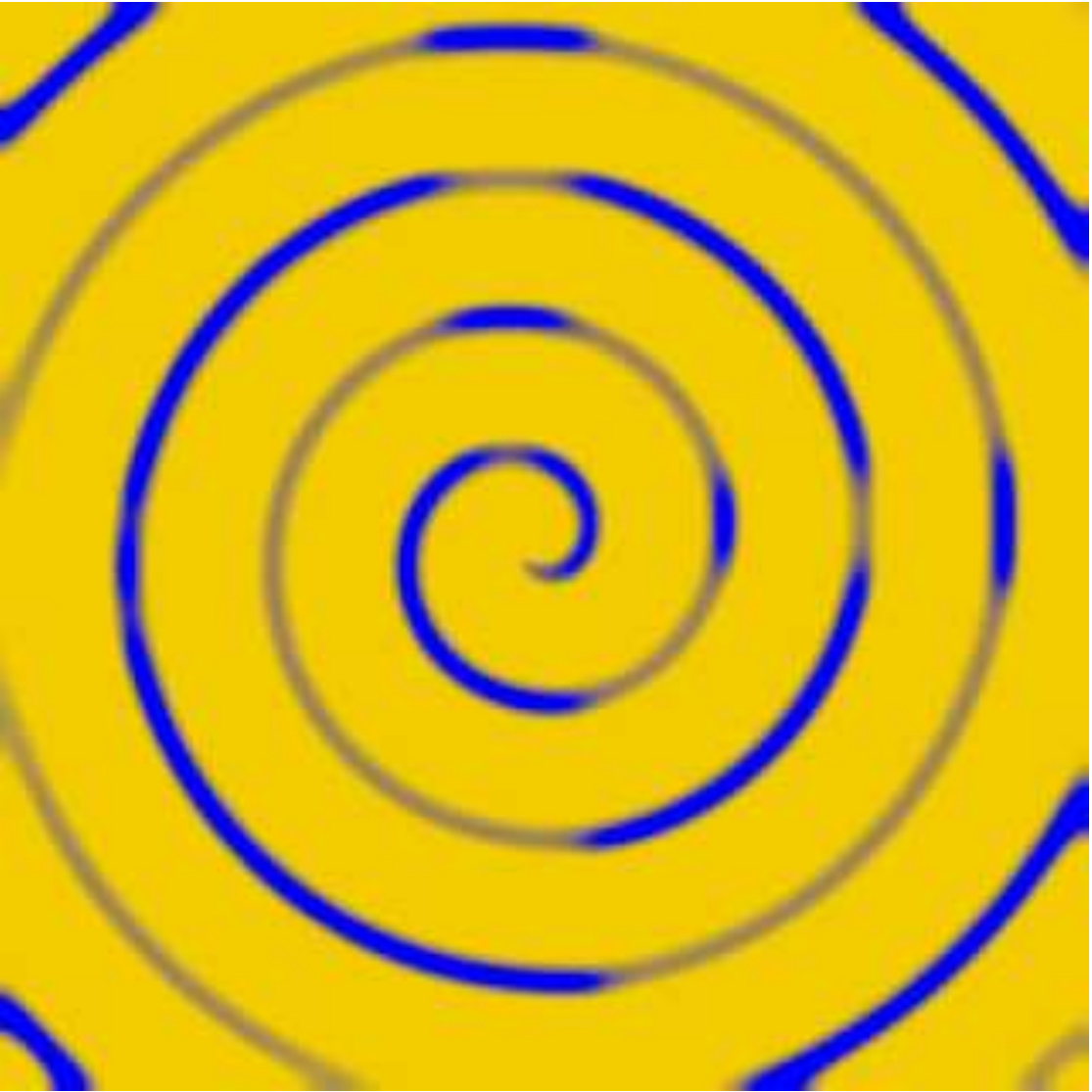}
		\caption{}\label{fig:spiral_defect_multiple}
	\end{subfigure}
	\begin{subfigure}[b]{0.3\textwidth}
		\centering  
		\includegraphics[height=0.6\textwidth]{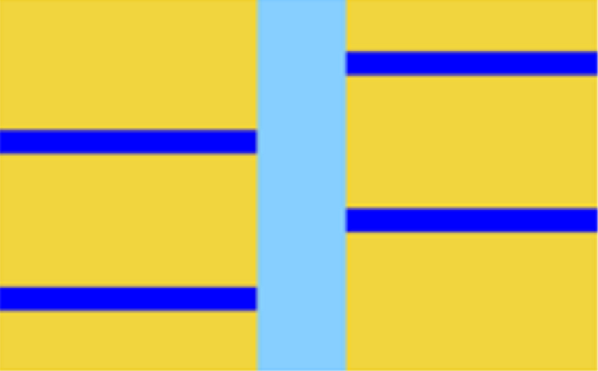} 
		\caption{}\label{fig:spiral_cut_abstracted1}
	\end{subfigure}
	\par
	\caption{Panel~(a) shows a snapshot of a rigidly-rotating planar spiral wave, while panels~(b) and~(c) contain snapshots of spiral waves that exhibit line defects (images taken from \citep{sandstede2007period}). Panel~(d) shows a space-time plot (space horizontal and time vertical) of a one-dimensional defect (shown in cyan) that mediates between two spatially homogeneous oscillations that are out of phase by half the period, thus representing one-dimensional versions of the line defects shown in panels~(b) and~(c): note the resemblance of this pattern with the pattern inside the red box shown in panel~(b).\label{fig:spiralwaves}}
\end{figure}

Our motivation comes from the line defects visible in the two center panels of Figure~\ref{fig:spiralwaves}, which are caused by the destabilization of a rigidly-rotating spiral wave through a period-doubling bifurcation \citep{sandstede2007period}: across the line defect, the phase of the spatio-temporal oscillations jumps by half a period. Over long time scales, these line defects attract and annihilate each other in pairs unless only one line defect is left: Figure~\ref{fig:spiral_defect_multiple} illustrates this behavior through the two pairs of co-located line defects that are about to merge and disappear. It is difficult to analyse these interaction properties between adjacent line defects in the full planar case, and we instead consider a simpler scenario in one space dimension that is more manageable. This scenario consists of a one-dimensional system that admits a point defect that mediates between two spatially homogeneous oscillations, whose phases jump by half their period across the defect: see Figure~\ref{fig:spiral_cut_abstracted1} for an illustration of the resulting space-time plots. A slightly different way to think about this scenario is to restrict the planar pattern with the line defect to the small red rectangle shown in Figure~\ref{fig:spiral_defect}: the resulting image resembles Figure~\ref{fig:spiral_cut_abstracted1}, and its time dynamics is similar.

In the one-dimensional case shown in Figure~\ref{fig:spiral_cut_abstracted1}, we could now concatenate several defects and attempt to understand their interaction properties. As a first step, we need to prove that we can actually truncate such a defect sitting at, say, $x=0$ from the entire line to a large bounded interval $(-L,L)$ supplemented by Neumann boundary conditions: once we know this, we can use reversibility or symmetry to create multiple copies by reflecting the truncated defect across $x=L$ or $x-L$. It is problem of establishing the existence of truncated defects on large intervals $(-L,L)$ that we will focus on in this paper. A different motivation for the same problem comes from validating numerical computations that are also conducted on bounded intervals rather than on the whole line. 

\subsection{Discussion of Defects}
\label{subsection:defect}
In this section, we will review the necessary definitions and results from the theory of one-dimensional defect patterns \citep{sandstede2004defects}. 
Consider the reaction-diffusion system 
\begin{equation}
	\label{eq:reaction_diffusion}
	u_t = D u_{xx} + f(u),~\mbox{ with }x\in \mathbb{R}, ~t\in \mathbb{R_+},~ u(x,t)\in 
	\mathbb{R}^d,~ f\in C^\infty (\mathbb{R}^d; \mathbb{R}^d),
\end{equation}
where $D$ is a constant, positive-definite diagonal matrix. Informally, \textit{defects} are  time-periodic solutions of  \eqref{eq:reaction_diffusion} that converge to spatio-temporally periodic structures as $x \to \pm \infty$. More formally, assume that $u_{wt}(k x - \omega t;k)$ is a family of solutions of \eqref{eq:reaction_diffusion} whose profiles are periodic in the first argument and parameterized by the \textit{wave number} $k$. These solutions are called \textit{wave trains}, and $ \omega$ is referred to as their \textit{frequency}. Typically, $\omega$ is uniquely determined by $k$ via the so-called \textit{nonlinear dispersion relation} $\omega = \omega_{nl}(k)$, and $u_{wt} = u_{wt}(k x - \omega_{nl}(k) t;k)$ is therefore a one-parameter family. Amongst the four types of generic defects, namely \textit{sources}, \textit{sinks}, \textit{transmission defects}, and \textit{contact defects} discussed in \citep{sandstede2004defects}, we focus here on \textit{contact defects}, which are typically symmetric under reflections in $x$ and resemble spatially homogeneous oscillations $u_{wt}(- \omega(0)t; 0)$ as $x \to \pm \infty$: they therefore reflect the pattern shown in Figure~\ref{fig:spiral_cut_abstracted1}. It will be useful to define $\omega_{nl}(0) =\colon \omega_d$ and use the rescaled time variable $\tau := \omega_d t$, so that the spatially homogeneous oscillations $u_{wt}(-\tau;0)$ are $2\pi$-periodic in $\tau$. With this notation, we can define contact defects more precisely.

\begin{definition}
	A function $u_d(x, \tau)$ is called a \textit{contact defect} with frequency $\omega_d$ if it is $2\pi-$periodic in $\tau$, satisfies the reaction-diffusion system
	\begin{equation}
		\label{eq:reaction_diffusion2}
		\omega_d u_\tau = D u_{xx} + f(u),~\mbox{ where }x\in \mathbb{R}, ~\tau\in \mathbb{R_+},~ u(x, \tau)\in 
		\mathbb{R}^d,~ f\in C^\infty (\mathbb{R}^d; \mathbb{R}^d),
	\end{equation}
	and, for some phase correcting functions $\theta_\pm(x)$ with $\theta_\pm'(x) \to 0$ as $x \to \pm \infty$, obeys
	\[
	u_d(x, \tau) - u_{wt}(- \tau  - \theta_\pm (x); 0)  \to 0  ~\mbox{ as } x \to \pm \infty.
	\]
	The convergence is assumed to be uniform in $\tau$ as $x \to \pm \infty$ for the functions and their first derivatives with respect to $x, t$.
\end{definition}

It is worth noting that the phase-correcting functions $\theta_\pm(x)$ will necessarily diverge logarithmically as $x \to \pm \infty$, see \citep[\S3.1]{sandstede2004evans}, which will pose difficulties later on as the phase of the defect does not converge to that of a single limiting wave train.

\begin{remark}
Contact defects were shown to exist in the complex cubic-quintic Ginzburg–Landau equation	\citep{sandstede2004defects}, and \citet[Theorem 17.17]{smoller1983shock} provided another existence result of contact defects as contact discontinuities.
\end{remark}

Our goal is to prove that contact defects persist under domain truncation to a sufficiently large interval $[-L,L]$ with suitable boundary conditions. 


\subsection{Main Results}

Before stating our persistence result, we reformulate the existence problem in terms of a spatial dynamical systems. We will state our hypotheses for the spatial dynamics problem rather than for the original reaction-diffusion system to keep the discussion concise and make it easier to connect the hypotheses more directly with the proofs in the later sections.

Since our focus is on time-periodic solutions, we proceed as in \citep{sandstede2004defects} and rewrite \eqref{eq:reaction_diffusion} as a first-order system 
\begin{align}
	\label{eq:spatial_dynamics}
	\begin{bmatrix}
		u_x \\ v_x
	\end{bmatrix} = \begin{bmatrix}
	v \\ - D^{-1}(-\omega u_\tau + f(u))
\end{bmatrix} =: G(u, v; \omega),
\end{align} 
with frequency $\omega$ near $\omega_d$, where the right-hand side is defined on the dense subspace $Y:= H^1(S^1) \times H^{1/2}(S^1)$ of $X:= H^{1/2}(S^1) \times L^2(S^1)$, and $S^1 := \mathbb{R}/2\pi \mathbb{Z}$ denotes the unit circle. In other words, we are  exchanging the evolution in time for evolution in the space variable $x$, hence the term "spatial dynamics". This method was pioneered by Kirchg\"assner \citep{kirchgassner1988nonlinearly, kirchgassner1981bifurcation} and Mielke \citep{mielke1996spatial}, see also \citep{sandstede2001structure, doelman2009dynamics, sandstede2004defects}. While the initial-value problem for \eqref{eq:spatial_dynamics} is ill-posed, many approaches from dynamical-systems theory, including invariant-manifold theory, continue to hold.

The system \eqref{eq:spatial_dynamics} is posed on Sobolev spaces on $S^1$, so there is a \textit{translation operator} $\mathcal{S}_\alpha :u(\tau) \to u(\tau + \alpha)$. The corresponding translation operator on $X$ will be denoted by $\mathcal{T}_\alpha =\mathcal{S}_\alpha \times \mathcal{S}_\alpha$. Given $B \subset X$, we will denote by $\Gamma(B) = \{\mathcal{T}_\alpha p|~ \alpha \in S^1, p \in B\}$ the union of the group orbits of the elements of $B$. 

We will use the notation  $\textbf{u}_d(x,\tau)  = (u_d, \partial_x u_d)$, and similar for $\textbf{u}_{wt}$. 
The wave train $u_{wt}(-\tau; 0)$, together with its $\tau$-translates, satisfies \eqref{eq:reaction_diffusion} when $\omega = \omega_d$, so $\textbf{u}_{wt}$ is an \textit{equilibrium} of \eqref{eq:spatial_dynamics}, and thus $\Gamma(\textbf{u}_{wt})$ is a \textit{circle of equilibria}. By definition, the contact defect $\textbf{u}_d(x, \tau)$ converges to $\Gamma(\textbf{u}_{wt})$ as $x\to\pm\infty$, and it is therefore a \textit{homoclinic orbit}. The circle of equilibria has center, stable, and unstable manifolds by \citep[Theorem~5.1]{sandstede2004defects}, and we use these to state our assumption that a contact defect exists.

\begin{hypothesis} \label{H1} 
Assume that $\textbf{u}_d(x, \cdot)\in W^{cs}(\Gamma(\textbf{u}_{wt}))$ satisfies \eqref{eq:spatial_dynamics} for $\omega=\omega_d = \omega_{nl}(0)$. We assume that  $\textbf{u}_d(x, \cdot) \not \in W^{ss}(\Gamma(\textbf{u}_{wt}))$. 
\end{hypothesis}

Our next hypothesis will be on the derivative $G_p(\textbf{u}_{wt} ; \omega_d)$\footnote{We will use the notation $G_p$ to denote the derivative of $G(u,v;\omega)$ with respect to $(u,v)$.}. One can readily check that $(\partial_\tau u_{wt}, 0)$ is an eigenvector and $(0, \partial_\tau u_{wt})$ a generalized eigenvector of the eigenvalue zero of $G_p(u_{wt},0 ; \omega_d)$. The eigenvector is generated by the $\mathcal{T}_\alpha$ symmetry by the circle group. We assume that there are no other eigenvalues, counted with multiplicity, on the imaginary axis so that $W^c(\Gamma(\textbf{u}_{wt}))$ has dimension two. 

\begin{hypothesis} \label{H2} 
We assume that zero is an eigenvalue of algebraic multiplicity two of $G_p(u_{wt},0 ; \omega_d)$ and that all other elements of the spectrum are bounded away from the imaginary axis. 
\end{hypothesis}

Besides $\tau$-symmetry, equation \eqref{eq:spatial_dynamics} has symmetry with respect to its evolution variable $x$. Recall that a \textit{reverser} of a dynamical system \citep{meiss2007differential} is a linear bounded involution such that $v(x):=\mathcal{R}\textbf{u}(-x)$ is a solution whenever $\textbf{u}(x)$ is a solution: alternatively, we can require that the reverser anti-commutes with the right-hand side of the dynamical system. The problem \eqref{eq:spatial_dynamics} has two reversers, namely the operators
\begin{align}
	\label{eq:reversers}
	\mathcal{R}_0:~(u,v)(\tau) &\to (u, -v)(\tau) \\
	\mathcal{R}_\pi\colon~(u,v)(\tau) &\to (u,-v)(\tau + \pi) = \mathcal{R}_0 \mathcal{T}_\pi (u,v).\nonumber
\end{align}

\begin{hypothesis}
	\label{H3}
	Assume that the defect $\textbf{u}_d$ is reversible, so that $\textbf{u}_d(0) \in \Fix \mathcal{R}$ where $\mathcal{R}$ is either $\mathcal{R}_0$ or $\mathcal{R}_\pi$.
\end{hypothesis}

By Hypothesis~\ref{H1}, we have $\textbf{u}_{d}(0) \in  W^{cs}(\Gamma(\textbf{u}_{wt}))$, and Hypothesis~\ref{H3}
implies that $\textbf{u}_{d}(0) \in  W^{cu}(\Gamma(\textbf{u}_{wt}))$, so that $W^{cs}$ and $W^{cu}$ intersect at the contact defect. By $\mathcal{T}_\alpha$ invariance, the intersection of these manifolds contains all time translates of the contact defect, and, generically, we do not expect it to contain anything else.

\begin{hypothesis}
	\label{H4} Assume that $W^{cs}(\Gamma(\textbf{u}_{wt}))$ and $W^{cu}(\Gamma(\textbf{u}_{wt}))$ intersect \textit{transversely} at $\textbf{u}_{d}(0)$, that is, the sum of their tangent spaces at each point $p \in \Gamma(\textbf{u}_{wt})$ is $X$. Our notation for transversality will be
	$W^{cs}(\Gamma(\textbf{u}_{wt})) \pitchfork W^{cu}(\Gamma(\textbf{u}_{wt}))$ at $\textbf{u}_{d}(0)$.
\end{hypothesis}

So far, our assumptions have been statements for the case $\omega = \omega_d$. 
When we change $\omega$, the circle of equilibria will disappear, and we will assume this is due to a non-degenerate saddle-node bifurcation.

\begin{hypothesis}
	\label{H5}
	We assume that the circle of equilibria undergoes a non-degenerate saddle-node bifurcation as we vary $\omega \approx \omega_d$.
\end{hypothesis} 

\begin{remark}
	\label{rk:cman}
	\citet[(8.15)]{doelman2009dynamics} show that, as we vary $\omega = \omega_d + \omega_*$, $\omega_* \approx 0$, the reduced vector field on the two-dimensional center manifold $W^c(\Gamma(\textbf{u}_{wt}))$ is of the form 
	\begin{align*}
		\alpha'(x) &= y,  \nonumber \\
		y'(x) &= -\frac{2\omega_*}{\lambda_{lin}''(0)} + \frac{\omega_{nl}''(0)}{\lambda_{lin}''(0)} y^2  + h.o.t.,
	\end{align*}
where $\alpha$ represents the coordinate given by time translation, $y$ is orthogonal to $\alpha$, $\lambda_{lin}$ is the linear dispersion relation, and $\omega_{nl}$ is the nonlinear dispersion relation. In particular, Hypothesis~\ref{H5} holds when $\lambda_{lin}''(0), \omega_{nl}''(0)$ are both nonzero.
\end{remark}

The following theorem is our main result.

\begin{theorem}{(Existence and uniqueness of truncated contact defects)} \label{theorem:existence}
Assume that Hypotheses~\ref{H1}-\ref{H5} hold, then there exist positive constants $\widehat{L}, C$ and a function $\epsilon_*:[\widehat{L}, \infty) \to (0, \infty)$ so that the following is true for each $L \geq \widehat{L}$. First, \eqref{eq:spatial_dynamics} with $\omega = \omega_d + \epsilon_*^2(L) $ has an $\mathcal{R}$-reversible  solution $\textbf{u}_L(x) = (u_L, u_L'):[-L;L] \to X$ that is uniformly at most $C/L^2$ away from $\Gamma(\textbf{u}_d(x))$ and satisfies the boundary conditions $\textbf{u}_L(\pm L)\in\Fix\mathcal{R}_0$. Furthermore, if $\textbf{u}_L$ and $\tilde{\textbf{u}}_L$ are two such solutions, then there exists an $\alpha \in S^1$ such that $\mathcal{T}_\alpha \textbf{u}_L(x) = \tilde{\textbf{u}}_L(x)$ for all $x$. Finally, the function $\epsilon_*(L)$ is $C^{2}$ and satisfies the estimates 
\begin{equation}
\label{eq:epsilon_dependency}
\epsilon_*(L) = \frac{2}{\pi L} +O\left(\frac{1}{L^{2}}\right), \qquad
\epsilon_*'(L) = \frac{-2}{\pi L^2} + O\left(\frac{1}{L^{3}}\right).
\end{equation}
\end{theorem}

We note that if $\mathcal{R}=\mathcal{R}_0$, then the truncated contact defects $\textbf{u}_L(x,\tau)$ extend to smooth $2L$-periodic solutions of \eqref{eq:spatial_dynamics}, since $\mathcal{R}_0$-reversibility of $\textbf{u}_L(x,\tau)$ together with $\textbf{u}_L(\pm L,\tau)\in\Fix\mathcal{R}_0$ implies $\textbf{u}_L(L,\tau)=\textbf{u}_L(-L,\tau)$. This is not true if the contact defect $\textbf{u}_d(x,\tau)$ is $\mathcal{R}_\pi$-reversible.

In order to prove Theorem~\ref{theorem:existence}, we will need the following auxiliary result on passage times through non-degenerate saddle-node bifurcations, which may be of independent interest.

\begin{theorem}\label{theorem:solving_epsilon}
Consider the system
\begin{equation}\label{eq:saddle}
y'(x) = \epsilon^2+    y^2 + g(y, \epsilon^2), 
\end{equation}
with parameter $\omega = \epsilon^2 \geq 0$, where  $g$ is $C^r$ for some $r \geq 4$ in both arguments, and  $g(0,0) = g_y(0,0) = g_{yy}(0,0) = g_{\omega}(0,0) = g_{y\omega}(0,0)= 0$, then the following is true.
\begin{enumerate}
\item There exist positive constants $\epsilon_0,\delta_0$ and a function $T=T(\epsilon; \delta)$ defined for $\epsilon \in (0, \epsilon_0]$ and $\delta\in [\delta_0/2; 2\delta_0]$ such that the solution of \eqref{eq:saddle} with $y(0) = - \delta$ satisfies $y(T(\epsilon, \delta)) = 0$.
\item There exists an $L_0 > 0$ and a unique function $\epsilon_*(L;\delta): (L_0, \infty) \times [\delta_0/2, \delta_0] \to (0, \epsilon_0)$, such that whenever $L \geq L_0$,
\[
L = T(\epsilon_*(L; \delta); \delta) \quad \mbox{ for all }L \geq L_0.
\]
For each fixed $\beta \in [0,1)$, the function $\epsilon_*$ is $C^{1 + \beta}$ in both arguments, and there is a $ C^{1,\beta}$ function $Q(z; \delta)$ such that
		\begin{align}
			\label{eq:epsilon_estimate}
			\epsilon_*(L;\delta) &= \frac{2}{\pi  L} +Q(L^{-1};\delta) = \frac{2}{\pi 
				L} +O(L^{-\beta - 1}), \\
			\frac{d\epsilon_*}{d L}(L;\delta) &= \frac{-2 }{\pi L^2} - \frac{ Q_z(L^{-1}; \delta)}{L^2} = 
			\frac{-2 }{\pi  L^2}  +O(L^{-\beta - 2}), \nonumber
		\end{align} 
where the constant in the big-O term may blow up as $\beta \to 1$. 
\item 	If, in addition, $g(-y,\omega)= g(y,\omega)$,  then $Q(\epsilon, \delta) \in C^{r}$, and the above estimates hold with $\beta = 1$. 
\item Analogous statements hold for the problem $y(0) = 0, y(T(\epsilon; \delta)) = \delta$.
\end{enumerate}
\end{theorem}

\subsection{Related work}

Theorem~\ref{theorem:existence} can be viewed as a result on the existence of periodic orbits with large periods near a given homoclinic orbit. Homoclinic bifurcations have been studied for many decades, and we refer to the survey \citep{HomburgSandstede} for references. Most results are for the case where the underlying equilibria are hyperbolic. Homoclinic bifurcations for nonhyperbolic equilibria have been considered for generic fold bifurcations, and we refer to \citep[\S5.1.10]{HomburgSandstede} for references. The case where homoclinic orbits approach a circle of equilibria with a two-dimensional center manifold was investigated first in the finite-dimensional case, and in fact for arbitrary Galerkin approximations of \eqref{eq:spatial_dynamics}, by the first author in \citep{ivanov2021truncation}. The proof in \citep{ivanov2021truncation} relies on the persistence of normally invariant manifolds for well-posed dynamical systems \citep{hirsch2006invariant, kuehn2015multiple}. Since similar results are not known for the infinite-dimensional ill-posed spatial dynamics problem considered here, we instead utilize Lin's method \citep{lin1990using} to prove Theorem~\ref{theorem:existence}.

Theorem~\ref{theorem:solving_epsilon} provides expansions of the travel from $y=-\delta$ to $y=0$ (and similarly from $y=\delta$ backwards in time to $y=0$) in the unfolding of a non-degenerate saddle-node bifurcation at $y=0$: our result shows that the travel times typically contain logarithmic terms $\log\epsilon$ are therefore not differentiable in $\epsilon$ regardless of how smooth the right-hand side is. In contrast, \citet{fontich2007general} considered the travel time from $y=-\delta$ to $y=\delta$ for the unfolding of a possibly degenerate saddle-node bifurcations: for analytic vector fields, they used the residue theorem to prove that the resulting travel times are analytic in $\epsilon$. These two results are reconciled by noting that the logarithmic terms in the travel times from $y=-\delta$ to $y=0$ and from $y=0$ to $y=\delta$ cancel, yielding a smooth expression for the travel times from $y=-\delta$ to $y=\delta$. We also note that we cannot assume analyticity since our results are needed for the vector field on a center manifold. Finally, we remark that  \citet{kuehn2008scaling} showed that travel times may exhibit many different scaling laws when the right-hand side depends only continuously on $\omega$.

The rest of the paper is organized as follows. In \S\ref{section:cman} we discuss the dynamics on the center manifold and prove Theorem~\ref{theorem:solving_epsilon}. Theorem~\ref{theorem:existence} is proved in \S\ref{section:proof_existence} using Lin's method, and we will provide  additional estimates on the truncated contact defect $u_L$ in \S\ref{sectiom:estimates}. We end with a brief discussion in \S\ref{section:conclusion}.

\section{Dynamics on the Center Manifold}
\label{section:cman}

Our goal in this section is to analyze the equations of the slow dynamics of a local equivariant center manifold near the circle of equilibria $\Gamma(\textbf{u}_{wt})$ using equivariant local coordinates $( \alpha, y)$. Henceforth, we will frequently use $\alpha$ as a coordinate of a two-dimensional center manifold, that corresponds to the drift along the group action $\mathcal{T}_\alpha$, and we will denote the coordinate, perpendicular to $\alpha$, by $y$. 

\citet[(8.15)]{doelman2009dynamics} show that, as we vary $\omega = \omega_d + \omega_*$, $\omega_* \approx 0$, the dynamics of \eqref{eq:spatial_dynamics} on the two-dimensional center manifold is of the form   
\begin{align*}
	\alpha'(x) &= y,  \nonumber \\
	y'(x) &= -\frac{2\omega_*}{\lambda_{lin}''(0)} + \frac{\omega_{nl}''(0)}{\lambda_{lin}''(0)} y^2  + O(|y|^3 +  |y\omega_*| + \omega_*^2);
\end{align*}
see Remark \ref{rk:cman}. 
In our situation, both reversers $\mathcal{R}=\mathcal{R}_0,\mathcal{R}_\pi$ act on the reversible local center manifold by $\mathcal{R}(\alpha, y) = (\alpha, -y)$, and the right-hand side of the $y$ equation is therefore even in $y$ for all sufficiently small $\omega_*$. Hence, up to rescaling by constant factors, we may assume the dynamics on the center manifold is 
\begin{align}
	\label{eq:center_manifold}
	\alpha'(x) &= y, \nonumber \\
	y'(x) &= \omega_* +  y^2  + g(y, \omega_*),
\end{align}
where $\alpha \in S^1$, $y \in [-2\delta_0, 2\delta_0]$ and $\omega_* \in [-\epsilon_0^2, \epsilon_0^2]$ (here $\delta_0, \epsilon_0$ are sufficiently small positive constants). We know $g(-y, \omega_*) = g(y, \omega_*)$, so in particular $g(y,\omega_*) = O(y^4 + \omega_*^2) $ and $g$ contains no $y, \omega_*y, y^3, \omega_*y^3$ terms in its Taylor expansion. In order to choose the value of $\omega_*$ in terms of the parameter $L$, we are going to need to need to study travel time in saddle-node bifurcations. We answer these questions in the next section, and we remark its results may be of independent interest.

\subsection{Passage Time Near Saddle Node Bifurcations}
\label{subsection:hitting_time}
The previous section shows that we need to study the dynamics of the saddle-node bifurcation $y'(x) = \omega_*+    y^2 + g(y, \omega_*)$, where $g(y, \omega_*) = O(y^3 + \omega_*^2)$ is a $C^4$ function and $g$ has other properties to be determined later. Whenever there are no equilibria (i.e. $\omega_* = \epsilon^2 > 0$), we want to answer the following questions:
\begin{enumerate}
	\item Given $\omega_*, \delta$, where $\delta \gg |\omega_*| > 0$ are sufficiently small, in what time does the solution of a saddle-node bifurcation travel between $y = 0$ and $y = \delta$? 
	\item Given a sufficiently large travel time $L > 1$ and a sufficiently small $\delta > 0$, can we find an $\omega_*$, such that the solution of a saddle-node bifurcation travels between $y=0$ and $y = \delta$ in time $L$? 
\end{enumerate}
We answer the first question in Lemma~\ref{lemma:hitting_time} and use it to answer the second question in Theorem~\ref{theorem:solving_epsilon}: 

The key result we need to prove Theorem~\ref{theorem:solving_epsilon} is Lemma~\ref{lemma:hitting_time} below, where we compute the time of flight from $0$ to $\delta$  in terms of  $\epsilon$. 

\begin{lemma} 
	\label{lemma:hitting_time}
	Consider the non-degenerate saddle-node bifurcation \eqref{eq:saddle} in the regime with no equilibria ($\epsilon > 0$) and with the same assumptions on $g$ as in Theorem~\ref{theorem:solving_epsilon}, then there exist numbers $\epsilon_0, \delta_0  > 0$ such that the following holds for all $\epsilon \in (0, \epsilon_0], \delta \in [\delta_0/2, 2\delta_0]$: 
	\begin{enumerate}
		\item There is an unique function $T_+(\epsilon, \delta)$, such that, the equation 
		\eqref{eq:saddle} with the initial condition $y(0) = 0$, satisfies $y(T_+(\epsilon, \delta)) = \delta$. We call this function $T_+$ the travel time between $0$ and $\delta$. 
		\item  There exist functions  $\eta(\epsilon) \in  C^r([0, \epsilon_0])$, $ \zeta(\epsilon, \delta) \in C^r([0, \epsilon_0] \times [\delta_0/2, 2\delta_0])$, such that $\eta(\epsilon) = O(\epsilon),  \zeta(\epsilon, \delta) = O(\epsilon)$ and
		\[
		\epsilon T_+(\epsilon, \delta) = \eta(\epsilon)\log \epsilon  + \frac{\pi}{2}  + \zeta(\epsilon, \delta). 
		\]
		In particular, $\epsilon T_+(\epsilon, \delta)$ is continuous up to $\epsilon = 0$, uniformly in $\delta$. 
		\item If, in addition, we assume $g(-y, \epsilon^2) = g(y, \epsilon^2)$ for all $y, \epsilon$, then the function $\eta(\epsilon)$ is identically zero, and then $\epsilon T_+(\epsilon, \delta) \in C^r([0, \epsilon_0] \times [\delta_0/2, 2\delta_0])$.  
	\end{enumerate}
\end{lemma}
\begin{proof}
	The idea of the proof is to construct an appropriate normal form for a saddle-node bifurcation and then to use partial fractions to compute the travel time.
	
	We start with the computation of the normal form. By \citep[Theorem~5 and Corollary~1]{ilyashenko1991finitely},
	each saddle-node bifurcation has the normal form 
	\begin{equation}
		\label{eq:saddle_normal}
		\tilde{y}'(x) = (\tilde{\omega}_*+    \tilde{y}^2)(1 + b(\tilde{\omega}_*)\tilde{y}), 
	\end{equation}
	where $b = b(\tilde{\omega}_*)$ is a $C^k$ function. We are interested in the case, where there is no $\tilde{y}$ term, so we will do the substitution $\tilde{y} = z + \zeta \tilde{\omega}_* b(\tilde{\omega}_*)$, where $\zeta$ will be determined later. This yields
	\begin{align*}
		z' &= \tilde{\omega}_* + \tilde{\omega}_* b(z + \zeta \tilde{\omega}_* b) + (z + \zeta \tilde{\omega}_* b)^2 + (z + \zeta \tilde{\omega}_* b)^3 b \\
		&= \tilde{\omega}_* + \zeta \tilde{\omega}_*^2 b^2 + \zeta^2 \tilde{\omega}_*^2 b^2 + \zeta^3 \tilde{\omega}_*^3 b^4 + z(\tilde{\omega}_* b + 2 \zeta \tilde{\omega}_* b + 3 \zeta^2 \tilde{\omega}_*^2 b^3) + z^2(1 + 3 \zeta \tilde{\omega}_* b^2) + z^3 b. 
	\end{align*}
	Therefore, we will pick $\zeta \approx -1/2$, so that the $z$ coefficient is zero, or
	\[
	3\zeta^2 \tilde{\omega}_* b^2 + 2 \zeta + 1 = 0, \quad \zeta = -1/2 - 3\tilde{\omega}_* b^2/8 + O(\tilde{\omega}_*^2 b^4).
	\]
	By the Inverse Function Theorem, we can rename the $\tilde{\omega}_* + \zeta \tilde{\omega}_*^2 b^2 + \zeta^2 \tilde{\omega}_*^2 b^2 + \zeta^3 \tilde{\omega}_*^3 b^4 $ as $\omega_*$, $3 \zeta \tilde{\omega}_* b(\tilde{\omega}_*)$ as $a(\omega_*)$ and $b(\tilde{\omega}_*)$ as $b(\omega_*)$, so that the saddle-node bifurcation equation takes the normal form
	\begin{equation}
		\label{eq:saddle_node}
		z' = \omega_* + z^2(1 + a(\omega_*)) + z^3 b(\omega_*).	
	\end{equation}	
	Per our computation, \eqref{eq:saddle_node} is a normal form of \eqref{eq:saddle}, so in fact there is a change of variables $y = \Psi(z; \omega_*)$, such that $\Psi'(0;\omega_*) = 1$, which converts \eqref{eq:saddle} to \eqref{eq:saddle_node}. Also, let $\tilde{\delta}$ be such that $\Psi(\tilde{\delta}, \omega_*) = \delta$ (of course, $\tilde{\delta}$ depends smoothly on $\omega_*$, but we suppress this in our notation.) The travel time of \eqref{eq:saddle} from $0$ to $\delta$ will be the same as the travel time of \eqref{eq:saddle_node} from $0$ to $\tilde{\delta}$, namely 
\[
T_+(\epsilon, \delta) = \int_0 ^{\tilde{\delta}} \frac{1}{ dz/dx } dz = \int_0 ^{\tilde{\delta}} \frac{1}{\epsilon^2 + z^2(1 + a(\epsilon^2)) + z^3b(\epsilon^2)}dz.
\]
The idea of the proof is to analyze the above integral via partial fractions. One obstacle to this approach is the fact that $\epsilon^2$ is small and $b$ might be zero, which obstructs the partial fraction decomposition. 
To remedy this issue, we multiply the integral by $\epsilon$ and then substitute $z = \epsilon/u$: 
\begin{align*}
	\epsilon T_+(\epsilon, \delta) &= \int_{ \epsilon/\tilde{\delta}} ^{ \infty} \frac{\epsilon}{\epsilon^2 + \frac{\epsilon^2}{u^2}(1 + a) \frac{\epsilon^3}{u^3}b} \frac{1}{u^2} du \\ &= \int_1 ^{\infty} \frac{u}{u^3 + u( 1 + a) + \epsilon b} du + \int_{\epsilon/\tilde{\delta}}^1  \frac{u}{u^3 + u( 1 + a) + \epsilon b} du  =: I_1 + I_2.
\end{align*}
By the Dominated Convergence Theorem $\epsilon T_+(\epsilon, \delta)|_{\epsilon = 0} = \pi/2$ (here we use the assumption  that $a(0) = 0)$. Again by the Dominated Convergence Theorem, the integral $I_1$ is as smooth in $\epsilon$ as the functions $a(\epsilon^2), b(\epsilon^2)$. For $I_2$ we use partial fractions: let 
$u_{1,2,3}$ be the roots of $u^3 + u( 1 + a) + \epsilon b$, where $u_1 = - \epsilon b + 
O(\epsilon^2 b^2)$, $u_{2,3}= \pm i + O(\epsilon b)$ and $u_2 = \bar{u}_3$. Then, there exist complex numbers $A_1, A_2, A_3$, such that
\[
\frac{u}{u^3 + u( 1 + a) + \epsilon b} = \frac{A_1}{u-u_1} + \frac{A_2}{u-u_2} + \frac{A_3}{u-u_3}.
\]
We can find $A_j$, $j = 1, 2, 3$ by multiplying the above equation by $u - u_j$ and then substituting $u = u_j$. This yields
\[
A_1 = \frac{u_1}{(u_1 - u_2)(u_1 - u_3) }= \frac{u_1}{\frac{d}{du}(u^3 + u( 1 + a) + \epsilon b)|_{u = u_1} } = \frac{u_1}{3 u_1^2 + 1 + a},
\]
and similar for $A_2, A_3$, so that
\begin{equation}
	A_j = \frac{u_j}{3 u_j^2 + 1 + a(\epsilon^2)}, \quad j = 1,2,3. 
\end{equation}
We analyze the sum of $A_2/(u - u_2)$ and $A_3/u - u_3$, where we use $u_2 = \bar{u}_3, A_2 = \bar{A}_3$:
\begin{align*}
	\frac{A_2}{u - u_2} + 	\frac{A_3}{u - u_3} &= \frac{A_2(u - u_3) + A_3(u - u_2)}{(u - u_2)(u - u_3)} = \frac{2 u \Re(A_2) - 2\Re (A_2 u_3)}{(u - 2 \Re u_2 u + (\Re u_2)^2 + (\Im u_2)^2} \\
	&= \frac{1}{\Im u_2} \frac{\frac{u - \Re u_2}{\Im u_2} B(\epsilon) + C(\epsilon)}{(\frac{u - \Re u_2}{\Im u_2})^2 + 1},
\end{align*}
where $B, C$ are $C^r$ functions of $\epsilon$, which can be computed explicitly from $u_2, u_3$. Therefore, integrating from $\epsilon/\tilde{\delta}$ to $1$ yields
\begin{align*}
	\int_{\epsilon/\tilde{\delta}}^1 \frac{A_2}{u - u_2} + 	\frac{A_3}{u - u_3} du &= \int_{\epsilon/\tilde{\delta}}^1 \frac{1}{\Im u_2} \frac{\frac{u - \Re u_2}{\Im u_2} B(\epsilon) + C(\epsilon)}{(\frac{u - \Re u_2}{\Im u_2})^2 + 1} du \\
	&= \frac{1}{2}B(\epsilon) \log \left( \left(\frac{u - \Re u_2}{\Im u_2}\right)^2 + 1\right) + C(\epsilon) \arctan \left(\frac{u - \Re u_2}{\Im u_2}\right) \Bigg|^1_{\epsilon/\tilde{\delta}},
\end{align*}
which are $C^r-$smooth in $\epsilon, \delta$ up to $\epsilon = 0$ (note that $\Im u_2 = 1 + O(\epsilon),$ so the denominators do not blow up). Therefore, the smoothness properties of $I_2$ are determined by $\int A_1/(u - u_1)$. In the case when $b(\epsilon^2) \equiv 0$, $u_1(\epsilon) = 0$, so $A_1(\epsilon) = 0$ and this term vanishes: this proves the third part of the theorem. If $b(\epsilon^2) \not\equiv 0$,
\begin{equation}
	\label{eq:saddle_tbu1}
	\int_{\epsilon/\tilde{\delta}}^1 \frac{A_1}{u-u_1}du =  \frac{u_1 \log(1 + 
		u_1)}{3u_1^2 + 1 + a(\epsilon^2)} - \frac{u_1 \log(\epsilon(1/\tilde{\delta} - u_1/\epsilon))}{3u_1^2 + 
		1 + a(\epsilon^2)} = - u_1 \log \epsilon + R(\epsilon, \tilde{\delta}), 
\end{equation}
where $R$ is a $C^r$ function. Therefore, we proved that $\eta(\epsilon) = - u_1(\epsilon)$ and this finishes the proof.
\end{proof}

\begin{corollary}
	\begin{enumerate}
		\item Under the same assumptions as Lemma~\ref{lemma:hitting_time}, the travel time from $- \delta$ to $0$ is given by 
		\[
		\epsilon T_-(\epsilon, \delta) = - \eta(\epsilon) \log \epsilon + \frac{\pi}{2} + \zeta_-(\epsilon, \delta),
		\]
		where $\eta(\epsilon)$ is the same as in Lemma~\ref{lemma:hitting_time} and $\zeta_-$ satisfies the same smoothness assumptions as $\zeta$.
		\item The travel time  $T_+(\epsilon, \delta) + T_-(\epsilon, \delta)$ from $- \delta$ to $\delta$ satisfies the following:
		\[
			\epsilon [T_+(\epsilon, \delta) + T_-(\epsilon, \delta)] = {\pi} + \zeta(\epsilon, \delta) + \zeta_-(\epsilon, \delta). 
		\]
		In particular the right-hand side is smooth in $\epsilon$ as $\epsilon \to 0$, even without the extra assumption about $g$ being even in $y$. 
	\end{enumerate} 
\end{corollary}
\begin{proof}
	The first part follows from the proof of the lemma, with the substitution $z \to -z$, $\tau = -t$. This will have the net effect of reversing the sign of $b$, while keeping $a$. Therefore, for the partial fraction decomposition, we would be looking for the roots of $u^3 + u(1+a) - \epsilon b$, and the first root $u_1^-$ will be $-u_1$. Therefore, in \eqref{eq:saddle_tbu1} we would see $+u_1 \log \epsilon$ instead of $-u_1 \log \epsilon$. 
	
	The second part of this corollary follows directly when we add the two travel times.  
\end{proof}

Now we can prove Theorem~\ref{theorem:solving_epsilon}, i.e. we solve for $\epsilon$ as a function of the total travel time $T_+(\epsilon, \delta)$.  
\begin{proof}
The above corollary shows 
	\begin{equation}
		\label{eq:passage_pf1}
		L =  T_+(\epsilon,\delta) =\frac{1}{\epsilon} \left[\frac{\pi}{2} + \eta(\epsilon) \log \epsilon + 
		\zeta(\epsilon; \delta)\right] 
	\end{equation}, where $\eta(\epsilon) = O(\epsilon), \zeta(\epsilon, \delta) = O(\epsilon)$. Taking $d/d\epsilon$ shows there is an 
	$\epsilon_0$, such that for $ \epsilon \in  (0 , \epsilon_0],$ the right-hand side is decreasing in $\epsilon$, hence bijective. We expect $\epsilon \approx 
	\pi/2L$, so we solve for $\pi/2L$: 
	\[
	\frac{\pi}{2L} = \frac{\epsilon}{1 + \frac{\pi}{2}\eta(\epsilon)\epsilon \log \epsilon + 
		\frac{\pi}{2}\epsilon \zeta(\epsilon, \delta)} =: \frac{\epsilon}{1 + W(\epsilon, \delta)}.
	\]
	We note that for all $\beta \in [0,1)$ $W(\epsilon,\delta)$ is  $C^{1 + \beta}$ in both arguments, as $\eta(\epsilon)\epsilon \log \epsilon \in C^{1 + \alpha}([0, \epsilon_0])$, and $W$ would be $C^r$ in both arguments if the  $\eta(\epsilon)\epsilon \log \epsilon$ term did not exist (e.g. for $g$ even). With $z:= \pi/(2L)$, we have 
	\[
	z = \frac{\epsilon}{1 + W(\epsilon, \delta)},
	\]
	so by the Implicit Function Theorem there exist constants $\epsilon_0, \kappa_1, \kappa_2$, such that, if  $\epsilon \in (-2\epsilon_0, 2\epsilon_0), z \in (\kappa_1, \kappa_2)$, the equation 
	has an unique solution $\epsilon_*(z, \delta)$.  
	By implicit differentiation
	\[
	\partial_z \epsilon_*(z; \delta_0) = \frac{1}{\partial_\epsilon \frac{\epsilon}{1 + W(\epsilon, \delta)}} = \frac{1}{1 + O(\epsilon_*^\beta)} = 1 + 
	O(\epsilon_*^\beta) = 1 + O(z^\beta).
	\]
	In the specific case $\eta(\epsilon) = 0$, we would obtain  $\epsilon_*'(z) = 1 + 
	O(z)$. Integrating in $z$ and recalling gives us
	\[\epsilon_* = z + O(z^{\eta+1}), \mbox{ or } \epsilon_*(z) =z + O(z^2) \mbox{ 
		in the case } \eta(\epsilon) = 0.
	\]
	Substituting $z = \pi/2L$ and defining $Q$ to be the remaining term finishes the proof. 
\end{proof}

It is  worth commenting on the size of the solutions for large $x$. Below we derive some estimates for the solutions of \eqref{eq:saddle}
\begin{lemma}
	\label{lemma:saddle_node_estimates}
	Fix $\beta \in [0,1)$. When $\epsilon=0$,	 the solution of \eqref{eq:saddle} has an asymptotic 
	expansion $y(x) = - 1/ x + O(x^{1 + \beta})$ as $x \to \pm \infty$. Furthermore, if $g_{yyy}(0,0) = 0$, the expansion is $y(x) = - 1/ x + O(x^{2})$. 
	
\end{lemma}

\begin{proof}
	We can assume $y(0) < 0$, so that $y(x) \to 0$ as $x \to \infty$. The proof in the other case is the same. Assume that $\delta$ is so small that when $|y| \leq \delta,$ $|g(y,0)| \leq C |y|^3 \delta 
	\leq |y|^2/2$. Then $ y' \in [|y^2|/2, 3|y|^2/2]$; since solutions of $y' = 
	y^2/2$ and $y' = 3y^2/2$ are both $O(1/x)$ as $x \to \infty$ we see that the 
	solution of \eqref{eq:saddle} is $O(1/x)$. 
	
	Now use this estimate in \eqref{eq:saddle} to get
	\[y' = y^2(1 + 
	O(x^{-1})),
	\] and the remainder would be $O(x^{-2})$ if $g_{yyy}(0,0) = 0$. We can solve this by separation of variables to obtain
	\[
	y(x)= \frac{1}{-x + O(\log x) + O(1)},
	\]
	where the $ O(\log x)$ term would not be present if $f_{yyy}(0,0) = 0$. 
	We can add $1/x$ to this equation and  obtain \[y(x) +  \frac{1}{x} = \frac{O(\log x) + O(1)}{x(- x +  O(\log x) + O(1))},
	\]
	so the right hand-side is $O(x^{-(1 + \beta)})$ for all $\beta \in [0,1)$. When $g_{yyy}(0,0) = 0$ there would be no logarithmic term, so we would just obtain $O(x^{-2})$.  
\end{proof}

\section{Existence of Truncated Contact Defects}
\label{section:proof_existence}

The main goal of this section is to prove Theorem~\ref{theorem:existence}, namely that we can truncate a contact defect to a large, bounded interval. The main geometric configuration in the case $\omega = \omega_d$ is presented in Figure \ref{fig:torus}. The proof formalizes the idea that, when we perturb $\omega$, the circle of equilibria $\Gamma(\textbf{u}_{wt})$ will disappear, but the invariant torus will persist and consist of the time translates of the truncated defects $\textbf{u}_L$. At the end of the section, we will explain in what sense the truncated contact defect is close to the original one. 

\begin{figure}[t]
	\includegraphics[scale=0.45]{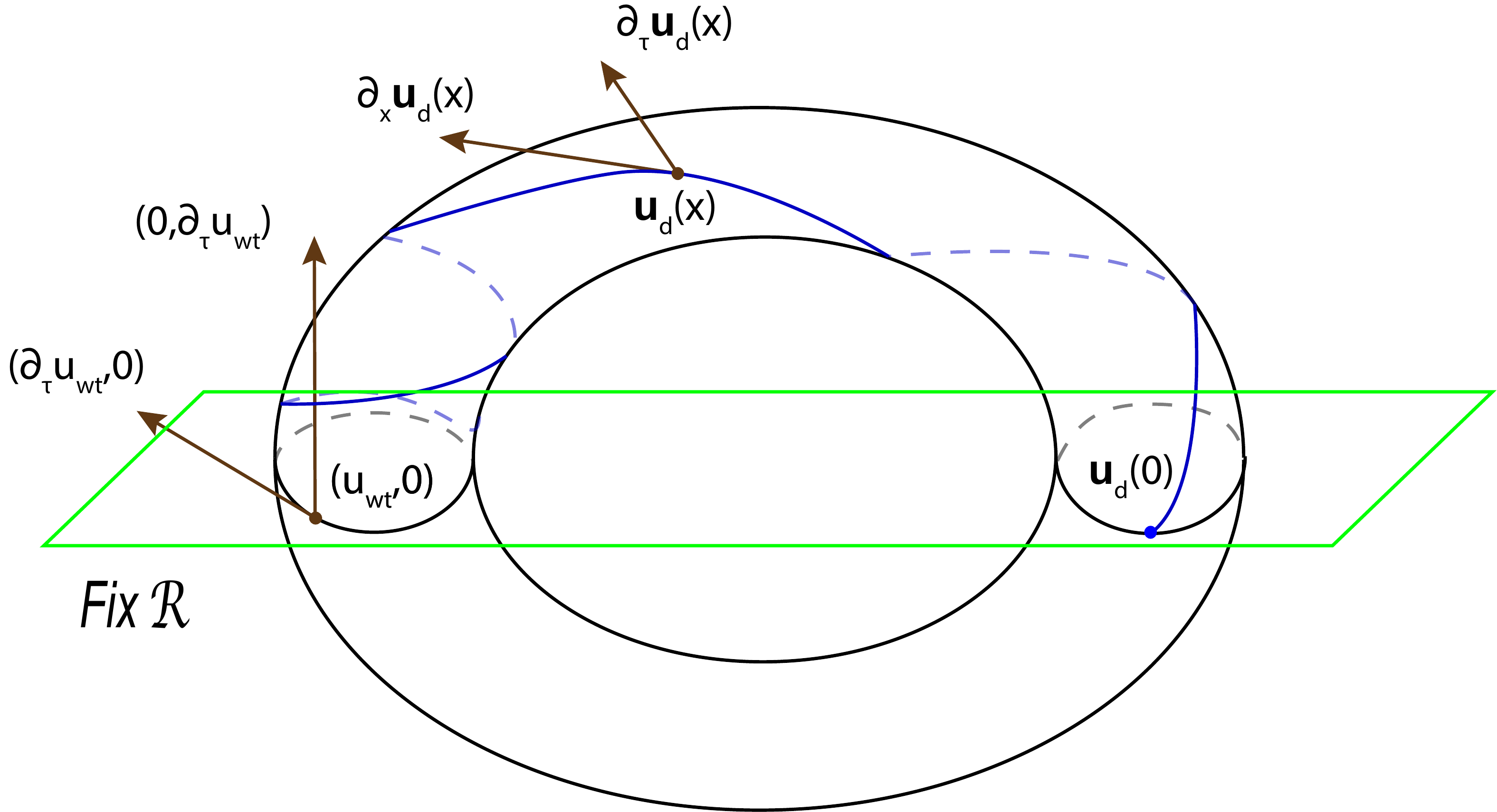}
	\centering
	\caption{The contact defect $\textbf{u}_d(x)$, in blue, converging to the circle of equilibria $\Gamma((u_{wt},0))$ for $x \geq 0$. For the sake of clarity, the analogous behavior as $x \to -\infty$ is not shown.}
	\label{fig:torus}
\end{figure}

\subsection{Exponential Trichotomies}
\label{section:trichotomies}

We first discuss exponential trichotomies of the linearization of \eqref{eq:spatial_dynamics} about the contact defect, which we will use to construct the truncated contact defects. Trichotomies allow us the decompose the underlying space into three complementary subspaces that contain, respectively, initial conditions of solutions that decay exponentially in forward time or backward time, or that grow only mildly. We note that Hypothesis~\ref{H2} shows that the linearization of \eqref{eq:spatial_dynamics} will have a two-dimensional center space, so we cannot expect that exponential dichotomies exist. The following theorem stating the existence of trichotomies was proved in \citep{sandstede2004defects}.

	\begin{theorem}
	\label{theorem:exponential_trichotomies}
	Assume  Hypotheses~\ref{H1}-\ref{H3}, then the linearization
\begin{align}
\label{eq:linearized_eq}
\begin{bmatrix}
u_x \\ v_x
\end{bmatrix} = \begin{bmatrix}
v \\ - D^{-1}(-\omega_d u_\tau + f'(u_d(x)))u
\end{bmatrix}
\end{align} 
of \eqref{eq:spatial_dynamics} about  $\textbf{u}_d(x)$ at $\omega = \omega_d$ has an exponential trichotomy on 
$\mathbb{R}$, that is, there exist strongly continuous families  $\{\Phi^{s}(\xi, \zeta)\}_{\xi, \zeta \in J, \xi \geq \zeta}$, $\{\Phi^{c}(\xi, \zeta)\}_{\xi, \zeta \in J, \xi \geq \zeta}$, $\{\Phi^{u}(\xi, \zeta)\}_{\xi, \zeta \in J, \xi \leq \zeta}$ of operators in $L(X)$ with the following properties:
\begin{enumerate}
\item $\Phi^j(\xi, \sigma)\Phi^j(\sigma, \zeta) = \Phi^j(\xi, \zeta)$ for $j = s, c, u$ and $\Phi^s(\xi, \xi) + \Phi^c(\xi, \xi) + \Phi^u(\xi, \xi) = 1$. 
\item There exist constants $C, \kappa > 0$, such that  
	\[ \|\Phi^s(\xi, \zeta) \| + \|\Phi^u( \zeta, \xi) \| \leq C \exp ( - \kappa|\xi - \zeta|) \]
	for all $\xi, \zeta$. Given $\eta \in (0, \kappa)$, there exists a constant $C(\eta)$, such that 
	\[
	\|\Phi^c(\xi, \zeta) \| \leq C(\eta)\exp (\eta |\xi - \zeta|).
\]
\item $\Phi^s(\xi, \zeta)\textbf{u}_0 $ and $\Phi^c(\xi, \zeta)\textbf{u}_0 $  satisfy \eqref{eq:linearized_eq} for $\xi > \zeta $ and  $\Phi^u(\xi, \zeta)$ satisfies \eqref{eq:linearized_eq} for
$\xi < \zeta$  whenever $\textbf{u}_0 \in Y, \xi, \eta \in J$. 
\end{enumerate}
\end{theorem}

We need reversibility (Hypothesis~\ref{H3}) to ensure the exponential trichotomies are defined on $\mathbb{R}$, otherwise we would only have exponential trichotomies on $\mathbb{R}^\pm$.  

The key feature of exponential trichotomies is roughness (see \citep[\S4]{coppel2006dichotomies} and \citep[Theorem~7.6.10]{henry2006geometric}), that is, sensitivity to perturbations of  \eqref{eq:linearized_eq}. Loosely speaking, exponential trichotomies persist when we perturb the solution $\textbf{u}_d$ we linearize about.

\subsection{Description of the center and center-stable manifolds}
\label{subsec:1}

In this section we describe the center-stable and center-unstable manifolds near $\textbf{u}_{wt}$. 

It is shown in \citep[Theorem 5.1]{sandstede2004defects} that in a neighborhood of the wave train $\Gamma{\textbf{u}_{wt}}$, \eqref{eq:linearized_eq} exhibits  center, center-stable manifolds $W^{cs}(\omega_*), W^c(\omega_*)$, smooth in the parameter $\omega_*$, and equivariant with respect to translation in $\tau$ ($\mathcal{T}_\alpha$). By Hypothesis~\ref{H2}, the center space has dimension two, and is spanned by $\partial_\tau \textbf{u}_{wt}, \partial_x \textbf{u}_{wt}$. Find a sufficiently small $\delta_0 > 0$, such that we can parameterize $W^c$ by local coordinates $\alpha \in S^1, y \in (-2 \delta_0, 2 \delta_0)$. 

In addition, $W^{cs}$ is parameterized by $W^c$ and strong-stable fibers $\mathcal{F}^{ss}(p, b^s, \omega_*)$, $p \in W^c$. We can find $\delta_1 > 0$, such that this fibration is valid for $|b| < \delta_1$, uniformly in $p = (\alpha,z) \in S^1\times(-2\delta_0, 2\delta_0)$. When $\omega_* = 0$, the contact defect $\textbf{u}_d(x) \in W^{cs}$  by Hypothesis~\ref{H1}, so we can find $L_0 \gg 1$, such that $u_d(L_0)$ belongs to a fiber of the point $p = (0, -\delta_0)\in W^c(0)$. The fibration in the case $\omega_* = 0$ can be seen on Figure \ref{fig:center_manifold}.

\subsection{Flow on the center manifold}
\label{subsec:2}
In this section we will describe the flow on $W^c(\omega_*)$, given by the equivariant coordinates $(\alpha, y)$. By \citep{sandstede2004evans}, the dynamics are given by \eqref{eq:center_manifold}. The normal form for the saddle-node bifurcation in $y$ was computed in \eqref{eq:saddle_node}, which we restate here: there are a coordinate transformation $y = \Psi(z)$, with $\Psi'(0) = 1$, and functions $a(\omega_*), b(\omega_*)$,
 such that the $y$ equation transforms to 
\[
	z' = \omega_* + z^2(1 + a(\omega_*)) + z^3b(\omega_*).
\]
However, by reversibility, the $y$ equation is symmetric with respect to the transformation $y \to -y$, so it must be the case that $b(\omega_*)=0$. Therefore, we obtain 
 \begin{equation}
	\label{eq:normal_form}
	z' = \omega_* + z^2(1 + a(\omega_*)).
\end{equation}
We will now outline some estimates on the travel time of solutions of the equations above. Namely, we look for solutions, which satisfy $y(-L) = -\delta_0, y(0) = 0$. 
\begin{lemma}
	\label{lemma:subsec2}
	Let $\omega_* = \epsilon^2 > 0$. 
	The following estimates hold for the solution of our saddle-node bifurcation equation \eqref{eq:saddle} with $y(0) = 0, y(-L) = - \delta_1$: 
	\begin{enumerate}
		\item If $|x|< 1/3$, $y(x) = \epsilon^2 x + O(\epsilon^4 x^2)$. 
		\item If $|x| < 1,$ $y(-L + x) = -\delta_0 + \delta_0^2( 1 + o(\delta_0))x + O(\epsilon |x| + \delta_0^3 x^2)$.  
	\end{enumerate}
\end{lemma}
\begin{proof}
	We will use the normal form \eqref{eq:normal_form}, $y = \Psi(z)$. 
	
	For the first inequality, let $X = \argmin \{|z(x)| = \epsilon^2\}$. Then 
	\[
	z(X) = \epsilon^2 \leq \int_0^x \epsilon^2 + 2 \epsilon^4 dy \leq 3 \epsilon^2 x, 
	\]
	so $X \geq 1/3$. By Taylor's theorem with integral remainder, 
	\[
	z(x) = z(0) + x z'(0) + \frac{1}{2}x^2 \int_0^1 z''(sx)ds = \epsilon^2 x + O(\epsilon^4 x^2).
	\]
	Therefore, by $y(x) = \Psi(z(x))$ with $\Psi'(0) = 1$, $y(x) = \epsilon^2 x + O(\epsilon^4 x^2)$. 
	
	The second inequality can be proven in a similar fashion. Define $\tilde{\delta}$ by $z(-L) = -\tilde{\delta}$ and do a Taylor expansion  
	\begin{align*}
	z(-L + x) &= z(-L) + z'(-L)x + \frac{x^2}{2} \int_0^1 z''(sx)ds \\&= - \tilde{\delta} + (\epsilon^2 + \tilde{\delta}^2(1 + a(\epsilon^2)))x + \frac{x^2}{2} O(2zz') = -\tilde{\delta} + \tilde{\delta}^2x + O(\epsilon^2|x| + x^2\tilde{\delta}^3).
	\end{align*}
By definition of a normal form transformation, $\Psi(z) = y$. and in particular $\Psi (-\tilde{\delta}) = - \delta_0$, so 
\begin{align*}
	y(-L+x) &= \Psi(z(-L + x)) = \Psi(-\tilde{\delta} + \tilde{\delta}^2x + O(\epsilon^2|x| + x^2\tilde{\delta}^3)) \\
	&= \Psi(- \tilde{\delta} + \tilde{\delta}^2 x ) + O(\epsilon^2|x| + x^2\tilde{\delta}^3) \\
	&= -\delta_0 + \Psi'(-\tilde{\delta})\tilde{\delta}^2x + O(\tilde{\delta}^4x^2) + O(\epsilon^2|x| + x^2\tilde{\delta}^3) \\
	&= -\delta_0 + \delta_0^2(1 + o(\delta_0))x +  O(\epsilon^2|x| + x^2\tilde{\delta}^3),
\end{align*}
where in the last line we used $- \delta_0 =\Psi(- \tilde{\delta}) = -\tilde{\delta}(1 + o(\tilde{\delta}))=-\tilde{\delta}(1 + o(\delta_0))$. 
\end{proof}

\subsection{Geometry near \texorpdfstring{$\Fix \mathcal{R}$}{Fix R}}
\label{subsec:3}
Assume that $u^c(x; \epsilon) \in W^c(\epsilon^2)$ and $\partial_x u^c(0; \epsilon) \neq 0$. Then, $\mathbb{R}u^c_x(0; \epsilon) \oplus \Ran (P^u(\textbf{u}_{wt}))\oplus \Fix \mathcal{R} = Y$. Indeed, this holds for $\epsilon = 0$, and the mapping $\mathcal{L}_\epsilon : \mathbb{R}u^c_x(0; \epsilon) \oplus \Ran (P^u(\textbf{u}_{wt}))\oplus \Fix \mathcal{R} \to Y$ is bijective and bounded when $\epsilon = 0$. Therefore, it is bijective and bounded uniformly for $\epsilon$ near $0$, and, by the Open Mapping Theorem, its inverse is uniformly bounded in $\epsilon$. 

\subsection{Pushforward of \texorpdfstring{$\Fix \mathcal{R}$}{Fix R}}
\label{subsec:4}
\begin{figure}[t]
	\centering
	\begin{subfigure}[b]{0.48\textwidth}
		\centering
		\includegraphics[width=\textwidth]{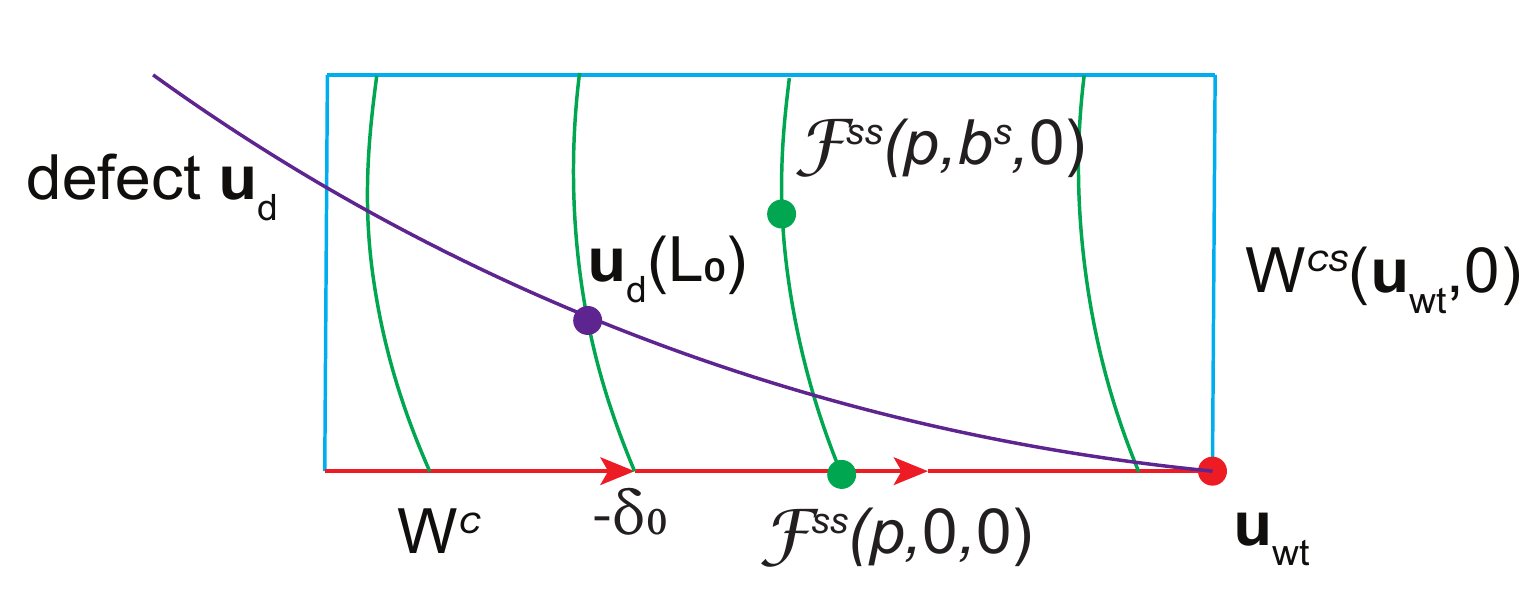}
		\caption{}\label{fig:center_manifold}
	\end{subfigure}
	\begin{subfigure}[b]{0.48\textwidth}
		\centering
		\includegraphics[width=\textwidth]{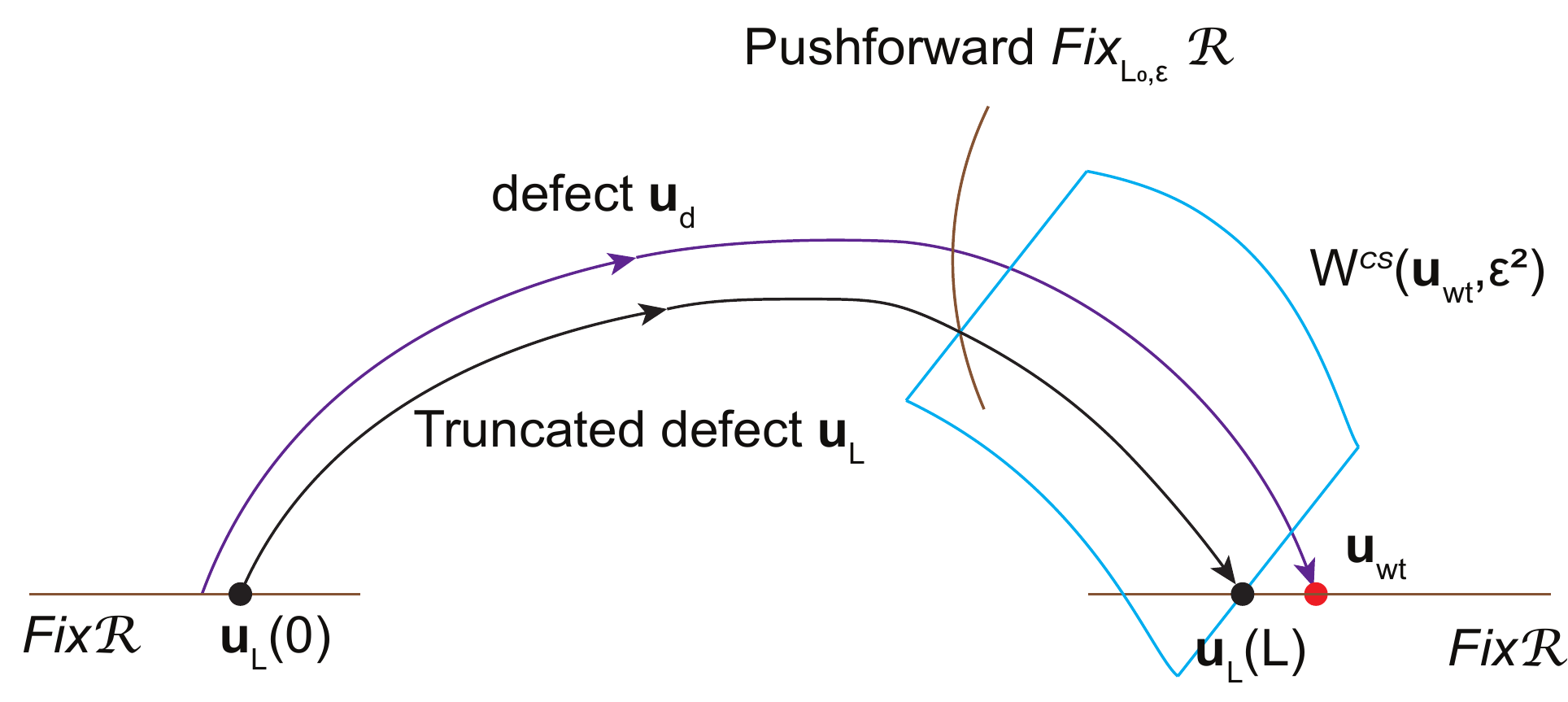}
		\caption{}\label{fig:dynamics_factored}
	\end{subfigure}
	\par
	\caption{Panel~(a) shows the fibration of the center-stable manifold and the defect when $\epsilon = 0$ (the time symmetry is factored out). Panel~(b) displays the intersection of the pushforward $\Fix_{L_0, \epsilon} \mathcal{R}$ and the center-stable manifold $W^{cs}(u_{wt}, \epsilon_0^2)$, and the corresponding solution $\textbf{u}_L$ (the time symmetry is factored out).  
	}
\end{figure}

The goal of this section is to compute the pushforward $\Fix_{L_0, \epsilon}\mathcal{R}$ of $\Fix \mathcal{R}$ along the defect $u_d$ from $x = 0$ to $x = L_0$ for each $L_0 \gg 1$ and each $\epsilon \ll 1$. Let $u(x) = u_d(x) + v(x)$, so $v_x = f_u(u_d(x)) v + O(|v|^2 + \epsilon^2)$. Let $\Phi_d^{c,s, u}(x,y)$ be an exponential trichotomy of the linearized about $u_d$ equation, and, additionally, let  $\mathcal{R}(\Ran \Phi^u_d(0,0)) = \Ran \Phi^s_d(0,0)$ (this is possible because the contact defect is reversible). 

\begin{lemma} \label{lemma:pushforward_Fix}
	For each $L_0 \gg 1$, $\epsilon \ll 1$, there exists a constant $C_1 > 0$, such that the pushforward $\Fix_{L_0, \epsilon} \mathcal{R}$ of $\Fix \mathcal{R}$ exists and  is parameterized by 
	\[
		\Fix_{L_0, \epsilon} \mathcal{R} = \{u_d(L_0) + a^u + O(e^{-\eta L_0}|a^u| + |a^u|^2 + \epsilon^2) ~:~ a^u \in \Ran \Phi^u_d(L_0, L_0) \mbox{ with } |a^u|, \epsilon \leq \frac{C_1}{L_0} \}.
	\]
\end{lemma}
\begin{proof}
	The idea of the proof is to use variation of parameters and the Banach Fixed Point theorem to construct the pushforward. We start with deriving the fixed-point equation. 
	Rewrite \eqref{eq:spatial_dynamics} with $v = u - u_d$: 
	\begin{align*}
		v_x &= G(u_d + v, \omega_*) - G(u_d, 0) = G_u(u_d, 0)v + (G(u_d + v, \omega_*) - G_u(u_d, 0)v - G(u_d, 0)) \\
		&=: G_u(u_d, 0)v + H(v,\omega_*), 
	\end{align*}
where $H(v;\omega_*) := G(u_d + v, \omega_*) - G_u(u_d, 0)v - G(u_d, 0)$. By Taylor's theorem\footnote{It will be convenient to write $\omega_* = \epsilon^2$},
\begin{align}
	\label{eq:cmt_est}
	H(v, \epsilon^2) &= O(|v|^2 +\epsilon^2) \\
	H_v(v, \epsilon^2) &= O(|v| + \epsilon^2). \nonumber
\end{align}
  Apply variation of parameters: 
	\begin{align}
		\label{eq:CMT_pushforward}
		v(x) &= \Phi^s_d (x,0)a^s + \Phi^u_d (x,L_0)a^u + \int_0^x \Phi^{cs}_d (x,y)H(v(y), \epsilon^2) dy \\
		&+ \int_{L_0}^x \Phi^{u}_d (x,y)H(v(y), \epsilon^2) dy, \quad 0 \leq x \leq L_0. \nonumber
	\end{align}
	By \eqref{eq:cmt_est} and the estimates for exponential trichotomies, the right-hand side of \eqref{eq:CMT_pushforward} is bounded by 
	\[
		C_0(|a^s| + |a^u| + L_0(\| v\|^2 + \epsilon^2)).
	\]
	With $C_1, C_2$ small, we will choose $|a^s|, |a^u|, \epsilon \leq C_1/L_0$ and $\| v \|_\infty \leq C_2/L_0$, so that
	\[
	\| RHS \| \leq C_0(2C_1 + C_2^2 + C_1^2)\frac{1}{L_0} \leq \frac{C_2}{L_0}
	\] 
	and, by \eqref{eq:cmt_est}: 
	\[
	\| D_v RHS \| \leq 2C_0 L_0 \|v\|_\infty \leq 2C_0 C_2 \leq \frac{1}{2}
	\]
	whenever $C_1, C_2$ are chosen small, depending on $C_0$, but not $L_0$. Therefore, by the Banach Fixed Point Theorem, there is an unique solution $v$ in the ball of radius $C_2/L_0$ for $|a^s|, |a^u|, \epsilon \leq C_1/L_0$, and 
	\[
	\| v\|_\infty \leq C_0 (|a^s| + |a^u| + L_0 \epsilon^2) \leq \frac{C_2}{L_0}. 
	\] 
	We impose the condition $v(0) \in \Fix \mathcal{R}$, so we can solve uniquely for $a^s = O(e^{-\kappa L}|a^u| + |a^u|^2 + \epsilon^2)$. Hence, there exists an unique $v(L_0)$, subject to $v(0) \in \Fix \mathcal{R}$, and it is given by 

	\begin{align*}
		v(L_0) &=  \Phi^s_d (L_0,0)a^s + a^u + \int_0^{L_0} \Phi^{cs}_d (x,y)H(v(y), \epsilon^2) dy \\
		&= a^u + O(e^{-\kappa L_0}|a^u| + |a^u|^2 + \epsilon^2)		
	\end{align*}
Therefore, the following holds:
\begin{equation*}
	\Fix_{L_0, \epsilon} \mathcal{R} = \{ u_d(L_0) + a^u + O(e^{-\kappa L_0}|a^u| + |a^u|^2 + \epsilon^2):~ a^u \in \Ran \Phi^u_d(L_0, L_0), \mbox{ with } |a^u|, \epsilon \leq \frac{C_1}{L_0} \}.
\end{equation*}
\end{proof}
The pushforward and the center-stable manifoldd are displayed on Figure \ref{fig:dynamics_factored}. The goal of the proof is to show that they intersect, and to adjust the parameter $\epsilon$ to ensure the resulting orbit travels from $\Fix \mathcal{R}$ to $\Fix \mathcal{R}$ in time $L$.  
\begin{remark}
In the above discussion we chose not to add a component in the $\partial_\tau$ direction, so it would not be incorrect to say the above result is on the pushforward of $\Fix \mathcal{R}/ \partial_\tau u_d(0)\mathbb{R}$. 	
\end{remark}

\subsection{Description of the center-stable manifold}
\label{subsec:5}
As noted in \S\ref{subsec:1}, the defect $\textbf{u}_d$ is in the center-stable manifold and $W^{cs}$ is fibered over $W^c$. In this section we will introduce notation for this fibration and we will express $\textbf{u}_d$ in said coordinates. 

We parameterize the strong-stable fibers in $W^{cs}$ with base points $p \in W^c$ as 
\begin{equation}
	\label{eq:fibration}
	\mathcal{F}^{ss}(p, b^s, \epsilon) = p + b^s + O(\delta_0|b^s|),
\end{equation}
where $b^s \in \Ran P^s(u_{wt})$ and $\mathcal{F}^{ss}(p, 0, \epsilon) = p \in W^c(\epsilon^2).$ In other words, $p$ is the base point of the fibration and $b^s$ parameterizes the fiber $\mathcal{F}^{ss}(p, b^s, \epsilon)$.

\begin{lemma}
	\label{lemma:subsec5}
	For all $\epsilon \geq 0, \epsilon \ll 1$, and all $L_0 \gg 1$, there is a base point $p_d(L_0, \epsilon) \in W^c$ and $b^s_d(L_0, \epsilon) \in \Ran P^u{u_{wt}},$ so that 
	\[
	\Fix_{L_0, \epsilon}\mathcal{R} \cap W^{cs}(\epsilon^2) = \mathcal{F}^{ss}(p_d(L_0, \epsilon), b^s_d(L_0, \epsilon), \epsilon). 
	\] 
\end{lemma}
\begin{proof}
	The intersection $\Fix_{L_0, \epsilon}\mathcal{R} \cap W^{cs}(\epsilon^2)$ is 
	the point $u_d(L_0)$\footnote{Note that, had we added the $\partial_\tau$ direction in \S\ref{subsec:4}, the intersection would have been a curve, instead of a point (but transversality would still hold).}. The intersection is transverse when $\epsilon = 0$: in this case, Lemma~\ref{lemma:pushforward_Fix} yields 
	\[\Fix_{L_0, 0} \mathcal{R} = \{u_d(L_0) + a^u + O(e^{-\eta L_0}|a^u| + |a^u|^2 ) ~:~ a^u \in \Ran \Phi^u_d(L_0, L_0) \mbox{ with } |a^u| \leq \frac{C_1}{L_0} \},
	\]
so the tangent space $T_{u_d(L_0)} \Fix_{L_0, 0} \mathcal{R}$ is $\Ran \Phi^u_d(L_0, L_0) + O(e^{-\eta L_0})$. The tangent space $T_{u_d(L_0)} W^{cs}$ is $\Ran \Phi^{cs}_d(L_0, L_0)$ by \citep[Theorem 5.1]{sandstede2004defects}, so indeed we have transversality when $\epsilon = 0$ and $L_0 \gg 1$.  Both $\Fix_{L_0, \epsilon}\mathcal{R} $ and $W^{cs}(\epsilon^2)$ are $C^1$ in $\epsilon$, so transversality persists when we perturb $\epsilon > 0$ by the stability theorem for transversality \citep[\S6]{guillemin2010differential}. 
\end{proof}
\begin{corollary}
	\label{cor:existence_u_eps}
	Assume $\epsilon > 0$, $\epsilon \ll 1$.
	There is an unique number $L(\epsilon)$ and unique $u^c(x; \epsilon) \in W^c$, such that $u^c(0; \epsilon) = 0$ and $u^c(-L; \epsilon) = p_d(L_0, \epsilon)$. Furthermore, $\epsilon L(\epsilon) \in C^1$ with $\epsilon L(\epsilon) = \pi/2 + O(\epsilon)$. In the notation we omit the dependence of $u^c$ and $L$ on $L_0$. 
\end{corollary}
\begin{proof}
	By Lemma~\ref{lemma:subsec5}, there is always a base point $p_d(L_0, \epsilon)\in W^c$ of $u_d(L_0, \epsilon)$. By the results in \S\ref{subsec:2}, the dynamics on the center manifold is determined by the $y$ equation; when $\epsilon > 0$, there is a finite travel time of $p_d(L_0, \epsilon)$ to $\Fix \mathcal{R}$ (i.e. to $\{y = 0\}$).   Lemma~\ref{lemma:hitting_time} shows that, as long as $L_0 \gg 1$ is fixed,  the travel time of $p_d(L_0, \epsilon)$ to $\Fix \mathcal{R}$ is  $L(\epsilon)$, such that $\epsilon L(\epsilon ) = \pi/2 + O(\epsilon)$ is $C^1$. 
\end{proof}
\subsection{Transversality of the pushforward }
\label{subsec:6}
We observe the following lemma holds:
\begin{lemma}
	$\Fix_{L_0, \epsilon} \mathcal{R}$ is transverse to $\mathbb{R}u^c_x(-L; \epsilon) \bigoplus \Ran P^s(u_{wt})$ near $u_d(0)$. 
\end{lemma}
\begin{proof}
	The proof follows from the transversality outlined in Lemma~\ref{lemma:subsec5}. 
\end{proof}
\subsection{Solving near the center-stable manifold}
\label{subsec:7}
The goal of this section is to apply variation of parameters to solve for orbits $u$ near the center-stable manifold. 

We start with some notation on fibrations. We will use the coordinates $p_d(L_0, \epsilon), b^s_d(L_0, \epsilon)$ from Lemma~\ref{lemma:subsec5} to define  
\[
\mathcal{F}^{ss}_R(b^s, \epsilon) := \mathcal{F}^{ss}(p_d(L_0, \epsilon), b^s + b^s_d(L_0, \epsilon), \epsilon).
\]
We will define $u_R(x; \epsilon, b^s)$ to be the solution such that $u_R (-L(\epsilon); \epsilon, b^s) = \mathcal{F}^{ss}_R(b^s, \epsilon)$, where $L(\epsilon)$ is given from  Corollary~\ref{cor:existence_u_eps}. In particular, for $\epsilon > 0$, $u_R (-L(\epsilon); \epsilon, 0)$ satisfies $u_R(0; \epsilon, 0) \in \Fix \mathcal{R}$. 
%
%
%
%
The variable $b^s$ will account for changes within the stable fiber (to be used later).

We will look for solutions of \eqref{eq:spatial_dynamics} of the type 
	\[
	u(x) = u_R(x; \epsilon, b^s) + v(x),
	\]
	so that $v_x = u(x)_x -u_R(x; \epsilon, b^s)  = G(u_R + v, \epsilon^2) - G(u_R, \epsilon^2)$, i.e 
	\begin{align} \label{eq:near_cman}
	v_x &= G_u(u_R(x, \epsilon, b^s), \epsilon^2)v + G(u_R + v, \epsilon^2) - G(u_R, \epsilon^2) - G_u(u_R(x; \epsilon, b^s), \epsilon^2)v \\ &=: G_u(u_R(x; \epsilon, b^s), \epsilon^2)v + H_R(v, \epsilon^2) \nonumber. 
	\end{align}
	By Taylor's theorem, 
	\begin{align} \label{eq:cmt_est_R}
		H_R(v, \epsilon^2) &= O(|v|^2) \\
		\partial_v H_R(v, \epsilon^2) &= O(|v|). \nonumber
	\end{align} 
Therefore, we will use the fixed-point equation
\begin{equation}
	\label{eq:fixed_point_center-stable}
	v(x) = \Phi^u_{\epsilon^2, b^s}(x,0)a^u_0 + \int_{-L + l_0}^x \Phi^{cs}_{\epsilon^2, b^s}(x,y)H_R(v(y), \epsilon^2) dy + \int_{l_1}^x \Phi^u_{\epsilon^2, b^s} (x,y)H_R(v(y), \epsilon^2)dy
\end{equation}
for $x \in [-L + l_0, l_1]$. 	A couple of remarks are in order. First, $\Phi^u_{\epsilon^2, b^s}$, $\Phi^{cs}_{\epsilon^2, b^s}$ come from the exponential trichotomies when linearizing about $u_R(x;\epsilon, b^s)$, hence they depend  on $\epsilon, b^s$, but by roughness of exponential dichotomies and trichotomies, the dependence is smooth and the bounds on exponential trichotomies can be chosen independently of $\epsilon, b^s$. Second, the reason why the equation for $v_x$ has no component like $\Phi^{cs}(x,0)a^{cs}_0$ is that here we aim to account for the unstable direction only, and we will use $b^s$ to account for the stable direction. The parameters $l_0, l_1$ are considered to be small; they need not be zero, because we will need them to match in the $\partial_x u_R$ direction near $\Fix_{L_0, \epsilon} \mathcal{R}$ and near $u_{wt}$ respectively. 

To apply the Contraction Mapping Theorem to \eqref{eq:fixed_point_center-stable} we will introduce the exponentially weighed norm $\|v\|_\eta := \sup_{x \in [-L + l_0, l_1] }e^{\eta |x|}|v(x)|$, where $\eta > 0$. As long as $\eta < \kappa$ ($\kappa$ is the exponent in the definition of exponential trichotomies), the following inequalities hold for the right-hand side of \eqref{eq:fixed_point_center-stable}: 

\begin{align*}
	\|RHS \|_\eta 
	&\leq C|a^u_0| + e^{\eta |x|} \left[   
	\int_{-L + l_0}^x e^{\epsilon(x-y)}e^{-2 \eta |y|} dy + \int_{l_0}^x e^{-\kappa(y - x)}e^{-2 \eta |y|} dy
	\right] C \|v\|^2_\eta \\
	&\leq C|a^u_0| + e^{\eta |x|} \left[  
	e^{-2\eta |x|} + e^{-2\eta (L -l_0 )}e^{\epsilon (L - l_0)} + e^{- \kappa |x|} + e^{-2\eta|x|}
	 	\right] C \|v\|^2_\eta \\
	 	&\leq C(|a^u_0| + \|v\|^2_\eta). 
\end{align*}
Similarly, we can show that$
\|\partial_v RHS \|_\eta \leq 1/2 $, so Banach's Fixed Point Theorem shows there is a unique solution $v(x) = v_*(x; \epsilon, \eta, l_0, l_1, L, b^s)$ and 
\[
|v(x)| \leq C e^{-\eta |x|}|a^u_0| 
\]
for all small $a^u_0 \in \Ran \Phi^u_R(-l_1, -l_1)$.  
In particular, we can estimate 
\[
v(-L + l_0) \leq Ce^{-\eta L}|a^u_0|,
\]
where $C$ is independent of $l_0, l_1$, so 
 \begin{equation}
 	\label{eq:sol_l0}
 	u_R(-L + l_0; \epsilon, b^s) + v(-L + l_0) = u_R(-L + l_0; b^s; \epsilon) + O(e^{-\eta L}|a^u_0|).
 \end{equation}

 Furthermore, our solution at $l_1$ is
\begin{equation}
	\label{eq:sol_l1}
	u_R(l_1; \epsilon, b^s) + v(l_1) = u_R(l_1;\epsilon, b^s) + a^u_0 + O(|a^u_0|^2).
\end{equation}
In the next section, we will need \eqref{eq:sol_l0}, \eqref{eq:sol_l1} to do the matching. 
\subsection{Matching}
\label{subsec:8}
Matching at $x = l_1$: by \eqref{eq:sol_l1} we have
\[
	u_R(l_1; \epsilon, b^s) + a^u_0 + O(e^{-\eta L} + |a^u_0|^2) \in \Fix \mathcal{R}. 
\]

Our matching at $x = -L + l_0$ looks like this: 

\[
	u_R(-L + l_0; \epsilon, b^s) + v(-L + l_0) \in \Fix_{L_0, \epsilon}\mathcal{R}, 
\]
and by Lemma~\ref{lemma:pushforward_Fix} and \eqref{eq:sol_l0}, 
 the condition at $x = -L + l_0$ is
\[
\phi_{l_0} (\mathcal{F}^{ss}_R(b^s,\epsilon)) + O(e^{-\eta x}|a^u_0|) = u_d(L_0) + a^u + O(e^{-\eta L_0}|a^u| + |a^u|^2 + \epsilon^2) 
\], 
where $\phi_{l_0}$ denotes the local flow on the center-stable manifold. By Taylor's theorem, this yields 
\[
\phi_{l_0}(\mathcal{F}^{ss}_R(0; \epsilon)) + b^s + O(\delta_0|b^s|) + O(e^{- \eta L}|a^u_0|) = \mathcal{F}^{ss}_R(0; \epsilon) + a^u + O(e^{-\eta L_0}|a^u| + |a^u|^2 + \epsilon^2)
\]
We will write the two matching conditions together an explain why they can be solved: 
\begin{align}\label{eq:matching_all}
	u_R(l_1; \epsilon, b^s) &+ a^u_0 + O(e^{-\eta L} + |a^u_0|^2) \in \Fix \mathcal{R}  \\
	\phi_{l_0}(\mathcal{F}^{ss}_R(0; \epsilon)) &- \mathcal{F}^{ss}_R(0; \epsilon) + b^s - a^u = O(\delta_0|b^s|) + O(e^{- \eta L}|a^u_0|) + O(e^{-\eta L_0}|a^u| + |a^u|^2 + \epsilon^2). \nonumber
\end{align}
The left-hand side of the first equation is a perturbation of a linear isomorphism $(l_1, a^u_0) \to \Fix \mathcal{R}$ by \S\ref{subsec:3}: varying $l_1$ corresponds to motion in the $\partial_x u_R(0;0,0)$ direction and varying $a^u_0$ allows one to traverse the remainder of $\Fix \mathcal{R}$, namely $\Fix \mathcal{R}/(\partial_x u_R(0;0,0) \mathbb{R})$. 
The left-hand side of the second equation is a perturbation of a boundedly invertible linear isomorphism as well (see \S\ref{subsec:6}: varying $l_0$ takes care of the motion in $\partial_x$ direction, and $b^s, a^u$ span the stable and unstable direction). Therefore, \eqref{eq:matching_all} is of the type 
$\mathcal{A} z = \mathcal{G}(z, \epsilon), G(z) = O(|z|^2 + |\epsilon|)$, where $z = (l_0, l_1, a^u, a^u_0, b^s)$ (here we are using Lemma~\ref{lemma:subsec2} to solve for $l_0, l_1$). The linear operator $\mathcal{A}$ is boundedly invertible by the arguments above, so such equations can be solved by the Banach Fixed Point Theorem. Therefore, $l_0, l_1, a^u, a^u_0, b^s$ are all parameterized by $\epsilon$. Finally, the solution, which we constructed, travels from $\Fix \mathcal{R}$ to $\Fix \mathcal{R}$ in time $L_0 + L(\epsilon) + l_0(\epsilon) + l_1(\epsilon))$, hence, if we want that travel time to be a fixed constant $L$, we can use the Implicit Function Theorem to solve for $\epsilon (L)$. To obtain the nonzero $\partial_\epsilon$ derivative, one can check that $\partial_\epsilon l_j(\epsilon) = O(\epsilon)$, $j = 0, 1$, and then use Corollary~\ref{cor:existence_u_eps}. This finishes the existence part of the proof. 

Finally, the uniqueness follows from the uniqueness in Banach's Fixed Point Theorem and by $\mathcal{T}_\alpha$ symmetry. 

\section{Estimates on Truncated Contact Defects}
\label{sectiom:estimates}

In this section, we estimate the distance between the truncated defect $\textbf{u}_L(x)$ to the original defect
$\textbf{u}_d(x)$ on $(-L,L)$. To do so, we can use the proof of Theorem~\ref{theorem:existence}. We remark that $\Gamma(\textbf{u}_d) \cup \Gamma(\textbf{u}_{wt})$, $\Gamma(\textbf{u}_L)$ are two invariant tori, which we proved are $O(\epsilon^2)$ away from each other. These tori inherit the local coordinates $(\alpha, y)$ and $(\alpha_L, y_L)$ from the center manifold, and we can extend these to global coordinates on the tori.  

\begin{corollary}
	Assume Hypotheses~\ref{H1}-\ref{H5}. 
	Assume $\textbf{u}_d(x), \textbf{u}_L(x)$ have local coordinates as described above. The 
	following estimates hold:  
	\begin{align}
		\label{eq:estimates_local_coord}
		|\omega_*(L)|  =|\epsilon_*^2(L)| &= O(L^{-2}), \nonumber \\
		\max_{|x|\leq L} |y(x) - y_L(x)| &=  O(L^{-1}), \nonumber \\
		\max_{|x|\leq L} |\alpha(x) - \alpha_L(x)| &=  O(\log L), \\
		\max_{|x|\leq L} |\alpha'(x) - \alpha'_L(x)| &=  O(L^{-1}) \nonumber.
	\end{align}
\end{corollary} 
\begin{proof}
	The first estimate in \eqref{eq:estimates_local_coord} follows from Theorem~\ref{theorem:existence}, where $\epsilon_*(L) = O(1/L)$. For any given $L_0$, 
	\[
		\max_{|x| \leq L_0}|y(x) - y_L(x)| +|\alpha(x) - \alpha_L(x)| = O(\epsilon^2),
	\]
	so we need to compute the maxima only over the interval $[L_0, L]$.
		
	By Lemma~\ref{lemma:saddle_node_estimates} we know $y(x) = O(1/x)$, so for $x \in [L_0, L],$ $y(x) = O(1/L)$. Over the same interval $|y_L(x)| < |y(x)| = O(1/L)$, so the second inequality in \eqref{eq:estimates_local_coord} follows. By  \eqref{eq:center_manifold} $\alpha' = y$, $\alpha_L' = y_L'$, so the fourth inequality is identical to the second one. When we integrate it, we obtain $\max_{|x| \leq L}|\alpha(x) - \alpha_L(x)| = 
	O(\log L)$.
\end{proof}

\section{Conclusion and Future Work}
\label{section:conclusion}

The present work answers in the affirmative the question of existence and uniqueness of truncated contact defects in reaction-diffusion systems. A forthcoming paper will address the issue of spectral stability of the constructed solutions under the assumption that the contact defect on the whole line is spectrally stable: it turns out that $\mathcal{R}_0$-reversible truncated contact defects are spectrally stable when periodic boundary conditions are used, while reversible truncated contact defects are always spectrally unstable under Neumann boundary conditions, regardless of which of the two reversers $\mathcal{R}_{0,\pi}$ is present, since the eigenvalue corresponding to the approximate eigenfunction $\partial_x \textbf{u}_L$ becomes positive. These results will, in particular, explain why these defect pairwise attract each other. We believe that nonlinear stability of contact defects and their truncation are difficult to establish due to the logarithmically diverging phase correction. Already in the case of source defects (whose spectrum appears to be "nicer" than the spectrum of contact defects \citep[Figure~6.1]{sandstede2004defects}), the proof of nonlinear stability is highly nontrivial \citep{beck2014nonlinear}.

\begin{Acknowledgment}
Ivanov was partially supported by the NSF under grant DMS-1714429.
Sandstede was partially supported by the NSF under grant DMS-1714429 and DMS-2106566. Ivanov was partially supported by the Ministry of Education and Science of Bulgaria, Scientific Programme "Enhancing the Research Capacity in Mathematical Sciences (PIKOM)", No. DO1-67/05.05.2022. 
\end{Acknowledgment}

\bibliographystyle{abbrvnat}
\bibliography{bibliography}

\end{document}